\documentclass[12pt, a4paper, reqno]{amsart}

\usepackage[utf8]{inputenc}
\usepackage{amsfonts,amssymb,amsxtra,amsthm,amsmath,amscd,mathrsfs,cite}
\usepackage[all]{xy}
\usepackage{enumerate}  
\usepackage{tikz}
\usepackage[normalem]{ulem}
\usepackage{soul}
\usepackage{color}

\setstcolor{red}

\makeatletter
\@namedef{subjclassname@2010}{\textup{2010} Mathematics Subject Classification}
\makeatother

\def \R {{\mathbb R}}

\renewcommand{\thefootnote}{\fnsymbol{footnote}}

\newtheorem{theorem}{Theorem}[section]
\newtheorem{lemma}[theorem]{Lemma}
\newtheorem{corollary}[theorem]{Corollary}

\newtheorem{proposition}[theorem]{Proposition}

\numberwithin{equation}{section}
\addtocounter{footnote}{1}

\def\rank {{\rm{rank}}}
\def \codim {{\rm{codim}}}

\begin{document}

\title[Multiple Gauss sums]{Multiple Gauss sums}

\author{Jianya Liu}
\address{Mathematical Research Center \& School of Mathematics, Shandong University, Jinan 250100, China}
\email{jyliu@sdu.edu.cn}

\author{Sizhe Xie}
\address{Mathematical Research Center, Shandong University, Jinan 250100, China}
\email{szxie@mail.sdu.edu.cn}

{\let\thefootnote\relax}


\maketitle

\begin{abstract}
A multiple Gauss sum is a complete multiple exponential sum twisted by Dirichlet characters. 
We prove a new bound for multiple Gauss sums and, as an application, improve previous results 
in the Birch--Goldbach problem. Let $F_1, \ldots, F_R \in \mathbb{Z}[x_1, \ldots, x_n]$ be forms 
with differing degrees, with $D$ being the highest degree, and let $\boldsymbol{F} = (F_1, \ldots, F_R)$ 
be nonsingular. We prove that the system $\boldsymbol{F}(\boldsymbol{x})=\mathbf{0}$ is solvable 
in primes provided that $n \geq D^2 4^{D+2} R^5$.   
\end{abstract}

\section{Introduction and statement of results} 

\subsection{Known results} 
Let $\boldsymbol{F}=(F_1,\ldots,F_R)$ be a system of forms, where $F_1, \ldots, F_R \in \mathbb{Z}[x_1, \ldots, x_n]$ 
are homogeneous polynomials with integer coefficients. Let $\boldsymbol{a} \in \mathbb{Z}^R $ and $q \in \mathbb{N} $ 
satisfy $(a_1,\ldots,a_R,q)=1$, and let $\boldsymbol{\chi}=(\chi_1,\ldots, \chi_n)$ be a system of Dirichlet characters modulo $q$. 
We study multiple Gauss sums defined by
\begin{align}\label{exp with chi}
        C_{\boldsymbol{F}}(q, \boldsymbol{a}; \boldsymbol{\chi})
        =\sum_{\boldsymbol{h} \bmod q}{\chi}_1(h_1) \cdots {\chi}_n(h_n)
        e\bigg(\frac{\boldsymbol{a}\cdot\boldsymbol{F}(\boldsymbol{h})}{q}\bigg).
\end{align}
Estimates for these sums are crucial in solving the Birch--Goldbach problem, which concerns solving 
the system of equations
\begin{align}\label{bfF=bf0}
\boldsymbol{F}(x_1, \ldots, x_n) = \mathbf{0}
\end{align}
in primes. Non-trivial bounds for $C_{\boldsymbol{F}}(q, \boldsymbol{a}; \boldsymbol{\chi})$ 
produce savings from finite places that, via the saving-transfer method, can be transferred to the 
infinite place, enabling successful treatment of enlarged major arcs in the circle method. For such 
applications see \cite{LiuXieSavTra}.

When $n=R=1$ and $F(x)=x$, the sum \eqref{exp with chi} reduces to the classical Gauss sum. 
For a one-variable monomial $F(x)=x^d$, Vinogradov \cite[Chap. 6, Exercise 14]{Vin} used the 
multiplicativity and periodicity of $\chi$ to obtain square root cancellation for
\begin{align*}\label{exp1/1/k} 
C_{x^d}(q, a; \chi) = \sum_{h=1}^q \chi(h) e\bigg(\frac{ah^d}{q}\bigg). 
\end{align*}
Cochrane and Zheng \cite{CocZhe} estimated sums for general one-variable polynomials $F(x)$ of degree $d$, proving that 
\begin{equation}\label{CocZhe}
|C_{F}(p^t, a; \chi)|\leq 4dp^{t(1-\frac{1}{d+1})} 
\end{equation} 
for any prime powers $p^t$. 
Very recently, Cochrane and Granville \cite{CocGra} improved \eqref{CocZhe} to 
\begin{align*}
|C_{F}(p^t, a; \chi)|\leq Cp^{t(1-\frac{1}{d})}, 
\end{align*} 
where $C$ is an absolute constant. 
In fact they established results for general summand $\chi(g(x))e(\frac{f(x)}{p^t})$ with $f,g$ rational functions over $\mathbb{Q}$.

The case \( n > 1 \) was first studied for prime moduli \( q = p \) by Fouvry and Katz \cite{FouKat} and by 
Fu \cite{Fu}, who obtained square-root savings of the form
\[
C_F(p, a; \boldsymbol{\chi}) \ll p^{\frac{n}{2} + \varepsilon},
\]
for any \( \varepsilon > 0 \).  For a survey on stratification methods in the study of exponential sums, see Bonolis, Kowalski, and Woo \cite{BonKowWoo}.

For general moduli \( q \), 
Fisher \cite{Fis} proposed an alternative strategy to separate $\chi(\cdot)$ from $e(\cdot)$, but it applies only to a 
limited class of polynomials. 
For general moduli \( q \) and a single form \( F \) of degree \( d \), Yamagishi \cite{Yam22} established the bound
\[
C_F(q, a; \boldsymbol{\chi}) \ll q^{n - \frac{n - \dim V_F^*}{2(2d-1)4^d} + \varepsilon},
\]
where \( V_F^* \) is the singular locus of the affine variety
\begin{align*}\label{def:V_F}
V_F = \{\boldsymbol{x} \in \mathbb{A}^n : F(\boldsymbol{x}) = 0\}.
\end{align*} 

Gauss sums in several variables can also be interpreted as trace functions; see 
Fouvry, Kowalski, and Michel \cite{FouKowMic} and Fouvry, Kowalski, Michel, 
and Sawin \cite{FouKowMicSaw} for ideas and results in this direction.

\subsection{Main results}   
We consider a system $\boldsymbol{F}=(F_1,\ldots,F_R)$ of forms in $n$ variables with differing degrees. 
Let $d$ be any degree appearing in the system and $r_d$ the number of forms of degree $d$. 
Write 
\begin{equation*}
\Delta:= \{d\in \mathbb{N} : \text{degree $d$ occurs in $\boldsymbol{F}$}\}.
\end{equation*}
For $d \in \Delta$ define the matrix
\[
J_{\boldsymbol{F}, d}(\boldsymbol{x}):=\left(\begin{array}{c}\nabla  F_{1,d}(\boldsymbol{x})\\ \vdots \\
\nabla F_{r_d,d}(\boldsymbol{x})\end{array}\right)
\]
and the affine algebraic variety
\[
S_d(n, \boldsymbol{F}):=\{\boldsymbol{x}\in\mathbb{A} ^{n}:\rank(J_{\boldsymbol{F}, d}(\boldsymbol{x}))<r_d\}. 
\]
Moreover, we set for $d \in \Delta$ that 
\begin{equation}\label{def Bd}
B_d(n, \boldsymbol{F}):=\dim S_d(n, \boldsymbol{F})
\end{equation}
in the sense of Birch. One sees that $B_d(n, \boldsymbol{F})<n$ for all $d$ if $\boldsymbol{F}$ is nonsingular.
When $r_d=0$, we shall take $B_d(n, \boldsymbol{F})=0$. 
For $1\le d\le D$ wirte
\begin{equation}\label{def sd}
    s_d(n, \boldsymbol{F}):=\sum_{i=d}^D \frac{2^{i-1}(i-1)r_i}{n-B_i(n, \boldsymbol{F})}.  
\end{equation}
One simply checks that $s_1(n, \boldsymbol{F})=\max_{d \in \Delta}s_d(n, \boldsymbol{F})$. 

Now we state our main results. 
In the following for a vector $\boldsymbol{a}= (a_1, \ldots, a_R)$, 
we will use the abbreviation that $\gcd\, (q, \boldsymbol{a}) = \gcd\, (q, a_1, \ldots, a_R)$. 

\begin{theorem}\label{prop gen Gau}
Let $\boldsymbol{F}=(F_1,\ldots,F_R)$ be a system of nonlinear forms in $\mathbb{Z}[x_1, \ldots, x_n]$ with highest degree $D$.   
Let \( \chi_1, \dots, \chi_n \) be Dirichlet characters modulo \( q \) respectively. 
For $\gcd \, (q, \boldsymbol{a})=1$ define $C_{\boldsymbol{F}}(q, \boldsymbol{a}; \boldsymbol{\chi})$ as in \eqref{exp with chi}.  
Then 
\begin{equation*}\label{bound gen Gau 1}
C_{\boldsymbol{F}}(q, \boldsymbol{a}; \boldsymbol{\chi}) 
\ll q^{n+\varepsilon}\min_{j \in \Delta} \bigg(\frac{\gcd\, (q, \boldsymbol{a}^{(j)}, \ldots,\boldsymbol{a}^{(D)}) }{q} \bigg)^{\frac{1}{4s_{2j}(4n, \boldsymbol{L})}},
\end{equation*}
with
\begin{equation*}
        \boldsymbol{L}(\boldsymbol{h},\boldsymbol{h}';\boldsymbol{j},\boldsymbol{j}')
        =\boldsymbol{G}(\boldsymbol{h};\boldsymbol{j})-\boldsymbol{G}(\boldsymbol{h};\boldsymbol{j}')-\boldsymbol{G}(\boldsymbol{h}';\boldsymbol{j})+\boldsymbol{G}(\boldsymbol{h}';\boldsymbol{j}')
\end{equation*}  
and $\boldsymbol{G}(\boldsymbol{h};\boldsymbol{j}) = \boldsymbol{F}(h_1j_1,\ldots,h_nj_n)$,
where, for any $j \in \Delta$, $\boldsymbol{a}^{(j)}=(a_{1,j},\ldots,a_{r_j,j})$  and $s_{2j}(4n, \boldsymbol{L})$ is defined in \eqref{def sd}.
\end{theorem} 

The proof is deferred to \S2. Although the statement of the theorem may appear intricate, 
its practical application will be demonstrated in the proof of Lemma \ref{lem/bound/gauss}. 
Furthermore, the theorem yields the following more transparent formulation.  

\begin{corollary}\label{thm2}
Under assumptions of Theorem \ref{prop gen Gau}, if $\boldsymbol{F}$ is nonsingular and $D$ is the 
highest degree in $\boldsymbol{F}$, then
\begin{equation*}\label{bound gen Gau 2}
C_{\boldsymbol{F}}(q, \boldsymbol{a}; \boldsymbol{\chi}) 
\ll q^{n+\varepsilon}\min_{j \in \Delta} \bigg(\frac{\gcd \, (q, \boldsymbol{a}^{(j)}, \ldots,\boldsymbol{a}^{(D)}) }{q} 
\bigg)^{\frac{n-R}{2^{2D+1}(2D-1)(R+1)R}}.
\end{equation*}
\end{corollary}

The proof of Corollary \ref{thm2} depends not only on Theorem \ref{prop gen Gau}, but also on the lemmas in \S3.  
Therefore, we postpone the proof of Corollary \ref{thm2} to \S3.  

When $\boldsymbol{F}$ has only one degree, we immediately get from Corollary \ref{thm2} a more explicit upper bound as follows.  

\begin{corollary}Under assumptions of Theorem \ref{prop gen Gau}, if $\boldsymbol{F}$ is nonsingular and $D$ is the 
only degree in $\boldsymbol{F}$, then
\begin{equation*}
C_{\boldsymbol{F}}(q, \boldsymbol{a}; \boldsymbol{\chi}) 
\ll q^{n+\varepsilon-\frac{n-R}{2^{2D+1}(2D-1)(R+1)R}}.
\end{equation*}
\end{corollary}

The main results presented above have direct implications for the Birch--Goldbach problem.  
Consider a system of forms   
$F_1,\dots,F_R \in \mathbb{Z}[x_1,\dots,x_n]$ with differing degrees, and let \(D\) denote the largest of these degrees.  
Let \(\boldsymbol{F}=(F_1, \dots, F_R)\) and assume that \(\boldsymbol{F}\) is nonsingular.  
Then, as will be established in Theorem \ref{thm:main}, the system  
$\boldsymbol{F}(\boldsymbol{x}) = \mathbf{0}$ 
is solvable in primes provided that $n \ge D^{2} 4^{D+2} R^{5}.$

\section{Proof of Theorem~\ref{prop gen Gau}} 

The proof uses the multiplicativity and periodicity of Dirichlet characters. 
The same idea was also applied in Vinogradov \cite[Chap. 6, Exercise 14]{Vin} and Yamagishi \cite[Lemma 7.4]{Yam22}. 

\begin{proof}
For any $\boldsymbol{j} \in \{(\mathbb{Z}/q\mathbb{Z})^\times\}^n $, we have 
\begin{align*}
        C_{\boldsymbol{F}}(q, \boldsymbol{a}; \boldsymbol{\chi}) 
        =\sum_{\boldsymbol{h} \bmod q}{\chi}_1(h_1j_1) \cdots {\chi}_n(h_nj_n)
        e\bigg(\frac{\boldsymbol{a}\cdot\boldsymbol{F}(h_1j_1,\ldots,h_nj_n)}{q}\bigg). 
\end{align*}
Summing over all $\boldsymbol{j}$ gives
\begin{align*}
\varphi^n(q) C_{\boldsymbol{F}}
&=\sum_{\boldsymbol{j}}\sum_{\boldsymbol{h} \bmod q} \chi_1(h_1)\chi_1(j_1) \cdots \chi_n(h_n)\chi_n(j_n)  
        e\bigg(\frac{\boldsymbol{a}\cdot\boldsymbol{F}(h_1j_1,\ldots,h_nj_n)}{q}\bigg) \\ 
&=\sum_{\boldsymbol{j}} \chi_1(j_1) \cdots \chi_n(j_n)\sum_{\boldsymbol{h} \bmod q}
\chi_1(h_1)\cdots \chi_n(h_n)
        e\bigg(\frac{\boldsymbol{a}\cdot\boldsymbol{G}(\boldsymbol{h};\boldsymbol{j})}{q}\bigg), 
\end{align*}
where $\boldsymbol{G}(\boldsymbol{h};\boldsymbol{j}) = \boldsymbol{F}(h_1j_1,\ldots,h_nj_n)$ 
and we used the multiplicative and periodic property of Dirichlet characters. By Cauchy's inequality,
\begin{align*}
\varphi^{2n}(q) |C_{\boldsymbol{F}} |^2
\leq \varphi^n(q) \sum_{\boldsymbol{j} \bmod q} \bigg| \sum_{\boldsymbol{h} \bmod q} 
\chi_1(h_1) \cdots \chi_n(h_n)
e\bigg(\frac{\boldsymbol{a}\cdot\boldsymbol{G}(\boldsymbol{h};\boldsymbol{j})}{q}\bigg) \bigg|^2. 
\end{align*}
The squared absolute value is
\begin{align*}
            =\sum_{\boldsymbol{h} \bmod q} \sum_{\boldsymbol{h}' \bmod q}
            \chi_1(h_1) \bar{\chi}_1(h_1') \cdots \chi_n(h_n)\bar{\chi}_n(h_n')
              e\bigg(\frac{\boldsymbol{a}\cdot(\boldsymbol{G}(\boldsymbol{h};\boldsymbol{j})-\boldsymbol{G}(\boldsymbol{h}';\boldsymbol{j}))}{q}\bigg), 
\end{align*}
so
\begin{align*}
\varphi^{2n}(q) |C_{\boldsymbol{F}} |^2 
\leq \varphi^n(q) 
            \ \sideset{}{^\ast} \sum_{\boldsymbol{h} \bmod{q}}
            \ \sideset{}{^\ast} \sum_{\boldsymbol{h}' \bmod{q}} 
            \bigg| \sum_{\boldsymbol{j} \bmod q} e\bigg(\frac{\boldsymbol{a}\cdot(\boldsymbol{G}(\boldsymbol{h};\boldsymbol{j})-\boldsymbol{G}(\boldsymbol{h}';\boldsymbol{j}))}{q}\bigg)  \bigg|.
\end{align*}
Applying Cauchy's inequality again yields
\begin{align*}
        \varphi^{4n}(q) |C_{\boldsymbol{F}}|^4 
        \leq \varphi^{4n}(q) \ \sideset{}{^\ast} \sum_{\boldsymbol{h} \bmod{q}}
            \ \sideset{}{^\ast} \sum_{\boldsymbol{h}' \bmod{q}} 
 \bigg| \sum_{\boldsymbol{j} \bmod q} e\bigg(\frac{\boldsymbol{a}\cdot(\boldsymbol{G}(\boldsymbol{h};\boldsymbol{j})-\boldsymbol{G}(\boldsymbol{h}';\boldsymbol{j}))}{q}\bigg)  \bigg|^2, 
\end{align*}
and hence
\begin{align}\label{no chi exp}
        |C_{\boldsymbol{F}} |^4 \leq \sum_{\boldsymbol{h} \bmod q} \sum_{\boldsymbol{h}' \bmod q} \sum_{\boldsymbol{j} \bmod q} \sum_{\boldsymbol{j}' \bmod q} e\bigg(\frac{\boldsymbol{a}\cdot\boldsymbol{L}(\boldsymbol{h},\boldsymbol{h}';\boldsymbol{j},\boldsymbol{j}')}{q}\bigg)
\end{align}
with
\begin{equation*}
        \boldsymbol{L}(\boldsymbol{h},\boldsymbol{h}';\boldsymbol{j},\boldsymbol{j}')
        =\boldsymbol{G}(\boldsymbol{h};\boldsymbol{j})-\boldsymbol{G}(\boldsymbol{h};\boldsymbol{j}')-\boldsymbol{G}(\boldsymbol{h}';\boldsymbol{j})+\boldsymbol{G}(\boldsymbol{h}';\boldsymbol{j}').
\end{equation*} 
The second application of Cauchy's inequality also guarantees the symmetry of variables. 
Note that the right-hand side of \eqref{no chi exp} contains no characters, so we can use results 
on complete exponential sums.
Here $\boldsymbol{L}$ is a system of the form 
$$
\{L_{i,2d}\}_{\substack{1\leq i \leq r_d \\ d \in \Delta}}
$$ 
in $4n$ variables.
By \cite[Lemma 8.2]{BroHB}, the right-hand side of \eqref{no chi exp} is
\begin{equation*}
 \ll q^{4n+\varepsilon}\min_{j \in \Delta} 
 \bigg(\frac{\gcd (q, \boldsymbol{a}^{(j)}, \ldots,\boldsymbol{a}^{(D)}) }{q} \bigg)^{\frac{1}{s_{2j}(4n, \boldsymbol{L})}}, 
\end{equation*}
where, for any $j \in \Delta$, $\boldsymbol{a}^{(j)}=(a_{1,j},\ldots,a_{r_j,j})$  and $s_{2j}(4n, \boldsymbol{L})$ is defined in \eqref{def sd}. 
This completes the proof. 
\end{proof}

\section{Geometric considerations}
For proving Corollary \ref{thm2} and applying Theorem \ref{prop gen Gau} in the Birch--Goldbach problem, we need some geometric considerations. 
\begin{lemma}\label{lem/biho}
    Let $f_1,\ldots, f_{r+1}\in \mathbb{A}[x_1,\ldots,x_n,y_1,\ldots,y_m] $ be bihomogeneous polynomials, that is  
    each $f_i$ is homogeneous in $\boldsymbol{x}$ and $\boldsymbol{y}$, respectively. 
    Let $X\subseteq \mathbb{P}^{n-1}\times \mathbb{P}^{m-1}  $ be defined by  $f_1,\ldots, f_{r}$ and  $Y\subseteq \mathbb{P}^{n-1}\times \mathbb{P}^{m-1}  $ be defined by  $f_1,\ldots, f_{r+1}$.
    Then
    \begin{equation*}
        \dim Y=\dim X \ \text{or} \ \dim X-1.
    \end{equation*}
\end{lemma}
\begin{proof}
 This follows from elementary properties of projective spaces.     
\end{proof}
Define the {\it singular loci} of the system $\boldsymbol{F}=(F_1,\ldots,F_R)$ as
\begin{equation*}
    V^{\ast}_{\boldsymbol{F}}=V^{\ast}_{\boldsymbol{F}}(n)=\{\boldsymbol{x}\in\mathbb{A} ^{n}:\rank(J_{\boldsymbol{F}}(\boldsymbol{x}))<R\}
\end{equation*}
in the sense of Birch. 
Then it is clear that
\begin{equation*}
    \dim V^{\ast}_{\boldsymbol{F}}(n)\le R
\end{equation*}
for a nonsingular system $\boldsymbol{F}$.
Set
$$
\codim V^{\ast}_{\boldsymbol{F}}=\codim V^{\ast}_{\boldsymbol{F}}(n)=n- \dim V^{\ast}_{\boldsymbol{F}}(n).
$$

The following proposition generalizes \cite[Theorem 5.1]{Yam19}.
\begin{proposition}\label{prop/dim}
Let $\boldsymbol{F}(\boldsymbol{x})$ be a system of $R$ forms $\in \mathbb{Z}[x_1,\ldots,x_n] $ 
whose degrees are all greater than $1$. 
Define a system of bihomogeneous forms
    \begin{equation*}
        \boldsymbol{G}(\boldsymbol{x};\boldsymbol{y})=\boldsymbol{F}(x_1y_1,\ldots,x_ny_n).
    \end{equation*}
    Then we have
    \begin{equation*}
       \min \{ \codim V_{\boldsymbol{G},1}^*, \codim V_{\boldsymbol{G},2}^*\} \geq \frac{\codim V_{\boldsymbol{F}}^*}{R+1},    
    \end{equation*}
    where 
\begin{equation*}
        V^*_{\boldsymbol{G},1}=\{(\boldsymbol{x},\boldsymbol{y}) \in \mathbb{A}^{2n}: \rank (J_{\boldsymbol{G},1})<R   \}
\end{equation*}
    with $J_{\boldsymbol{G},1}$ being the first $n$ columns of the Jacobian matrix $J_{\boldsymbol{G}}$ of $\boldsymbol{G}$, 
 and $J_{\boldsymbol{G},2}$ being the last $n$ columns of the Jacobian matrix $J_{\boldsymbol{G}}$ of $\boldsymbol{G}$. 
\end{proposition}
\begin{proof}
    In fact, most of the argument can be directly copied from that of \cite[Theorem 5.1]{Yam19}, except for the part that has essential differences. 
    However, for the sake of completeness, we will rewrite it with apppropriate omissions.
    
    Let $X$ be an irreducible component of $V_{\boldsymbol{G},1}^*$ with $\dim X= \dim V_{\boldsymbol{G},1}^*$.
    Up to reordering of variables we may assume that
    $$
    X \nsubseteq V(y_j) \ (1\leq j\leq m) \ \text{and} \ X\subseteq V(y_i) \ (m+1\leq j \leq n)
    $$
    for some $0\leq m \leq n$.

    Claim 1: There exists $(z_1,\ldots,z_m) \in (\mathbb{C}\setminus \{0\} )^m$ such that
    $$
    \dim X \cap (\cap_{1\leq j \leq m} V(y_j-z_j)) \geq \dim X-m.
    $$
    The proof of Claim 1 is just the same as that in \cite[Theorem 5.1]{Yam19}.
    Let $z_{m+1}=\ldots=z_n=0$.
    Then we have 
    \begin{equation}\label{fir/step}
        \begin{split}
            \dim X \cap (\cap_{1\leq j \leq n} V(y_j-z_j))    
            &=\dim X \cap (\cap_{1\leq j \leq m} V(y_j-z_j)) \\
            &\geq \dim X-m \\
            &=\dim V_{\boldsymbol{G},1}^* -m.   
        \end{split}
    \end{equation}
    We also have
    \begin{equation}\label{transition}
        X \cap (\cap_{1\leq j \leq n} V(y_j-z_j)) \subseteq V_{\boldsymbol{G},1}^* \cap (\cap_{1\leq j \leq n} V(y_j-z_j)). 
    \end{equation}
    For each $1\leq k \leq n$, we define
    \begin{equation}\label{def/Mk}
       M_k=
       \begin{pmatrix}
       \frac{\partial{F_1}}{\partial{x_1}}(\boldsymbol{x}) & \ldots & \frac{\partial{F_1}}{\partial{x_k}}(\boldsymbol{x}) \\
       \ldots & \ldots & \ldots \\
       \frac{\partial{F_R}}{\partial{x_1}}(\boldsymbol{x}) & \ldots & \frac{\partial{F_R}}{\partial{x_k}}(\boldsymbol{x})
       \end{pmatrix}
       =
       \begin{pmatrix}
        \boldsymbol{M_{k,1}}  \\
        \ldots  \\
        \boldsymbol{M_{k,R}}
       \end{pmatrix}
    ,
    \end{equation}
    \begin{equation}\label{def/Tk}
        T_k=\{ \boldsymbol{x}\in \mathbb{A}^n: \rank\, M_k <R, x_{k+1}=\ldots=x_n=0 \}
    \end{equation}
    and
    \begin{equation}\label{def/Uk}
        U_k=\{ \boldsymbol{x}\in \mathbb{A}^n: \rank\, M_k <R, x_{k+2}=\ldots=x_n=0 \}.
    \end{equation}
    Here $T_k$ and $U_k$ are affine varieties.
    Then it is clear that $T_n=V_{\boldsymbol{F}}^*$ and $\dim T_{k+1} \leq \dim U_k=\dim T_k$ or $\dim T_k +1$ as affine varieties.
By \eqref{transition}, \eqref{def/Mk} and \eqref{def/Tk} we obtain
    \begin{equation}\label{sec/step}
        \dim (X \cap (\cap_{1\leq j \leq n} V(y_j-z_j))) \leq n-m+ \dim T_m.
    \end{equation}
    
    Claim 2: We have 
    \begin{equation}\label{Tk/upp/bound}
        \max_{1\leq k \leq n} \{\dim T_k \} \leq \frac{Rn+\dim V_{\boldsymbol{F}}^*}{R+1}.
    \end{equation}
    It is worth mentioning that there are significant differences between our proof for Claim 2 and that in \cite[Theorem 5.1]{Yam19}, which stems from the distinction between a single form and a system of forms.  
    And we need more delicate discussions.

    The crucial part is to give a nice upper bound for $\dim U_k - \dim T_{k+1} $.
    Put, by \eqref{def/Mk}, 
    \begin{equation}\label{def/Xkk+2}
        X_{k,k+2}=\{(a_1,\ldots,a_R, \boldsymbol{x}) \in \tilde{\mathbb{A}}:a_1\boldsymbol{M_{k,1}}+\ldots+ a_R\boldsymbol{M_{k,R}}=\boldsymbol{0}, x_{k+2}=\ldots=x_n=0  \}
    \end{equation}
    and
    \begin{equation}\label{def/Xk+1k+2}
        X_{k+1,k+2}=\{(a_1,\ldots,a_R, \boldsymbol{x}) \in \tilde{\mathbb{A}}:a_1\boldsymbol{M_{k+1,1}}+\ldots+ a_R\boldsymbol{M_{k+1,R}}=\boldsymbol{0}, x_{k+2}=\ldots=x_n=0  \},
    \end{equation}
    where $\tilde{\mathbb{A}}:=(\mathbb{A}^R \setminus \{\boldsymbol{0}\}) \times (\mathbb{A}^n \setminus \{\boldsymbol{0}\})$.
    Consider the canonical maps $\tilde{\mathbb{A}}\to \mathbb{P}^{R-1}\times \mathbb{P}^{k}$ and ${A}^{k+1} \setminus \{\boldsymbol{0}\} \to \mathbb{P}^k $, and denote by $\tilde{X}_{k,k+2}$, $\tilde{X}_{k+1,k+2}$, $\tilde{U}_k$ and $\tilde{T}_{k+1}$ the images of $X_{k,k+2}$, $X_{k+1,k+2}$, $U_k$, $T_{k+1}$, respectively.
    Then the projection map $\pi: \mathbb{P}^{R-1}\times \mathbb{P}^{k} \to \mathbb{P}^{k}$ induces two surjective and regular maps $\pi_{k,k+2}: \tilde{X}_{k,k+2}\to \tilde{U}_k $, by \eqref{def/Xkk+2} and \eqref{def/Uk}, and $\pi_{k+1,k+2}: \tilde{X}_{k+1,k+2}\to \tilde{T}_{k+1}$, by \eqref{def/Xk+1k+2} and \eqref{def/Tk}.
    
    Choosing an irreducible component $\tilde{X}^0_{k,k+2}$ of $\tilde{X}_{k,k+2}$ with $\dim \tilde{X}^0_{k,k+2}=\dim \tilde{X}_{k,k+2}$, we have $\pi_{k,k+2}(\tilde{X}^0_{k,k+2})$ is irreducible.  
    By \cite[Corollary 11.13]{Har} we get 
    \begin{equation*}
        \begin{split}
        \dim \tilde{X}^0_{k,k+2} - \min_{p \in \pi_{k,k+2}(\tilde{X}^0_{k,k+2}) } \{ \dim \pi_{k,k+2}^{-1}(p)  \}
        &=\dim \pi_{k,k+2}(\tilde{X}^0_{k,k+2}) \\
        &\le \dim \tilde{U}_k.
        \end{split}
    \end{equation*}
    Since we have the trivial bound $\dim \pi_{k,k+2}^{-1}(p)\leq R-1$ for any $p$, it follows that
    \begin{equation*}
        \dim \tilde{X}_{k,k+2} - (R-1) \leq \dim \tilde{U}_k \leq \dim \tilde{X}_{k,k+2},
    \end{equation*}
    by the surjectivity of $\pi_{k,k+2}$.
    Similarly, we can get
    \begin{equation*}
        \dim \tilde{X}_{k+1,k+2} - (R-1) \leq \dim \tilde{T}_{k+1} \leq \dim \tilde{X}_{k+1,k+2}.
    \end{equation*}
    Then it follows from Lemma \ref{lem/biho} that $\dim \tilde{X}^0_{k+1,k+2}=\dim \tilde{X}^0_{k,k+2}$ or $\dim \tilde{X}^0_{k,k+2}-1$.
    Therefore we deduce from the above that
    \begin{equation*}
        \dim U_k - \dim T_{k+1}
        =\dim \tilde{U}_k-\dim \tilde{T}_{k+1} \leq R.
    \end{equation*} 
    Recall that $\dim T_{k+1} \leq \dim U_k=\dim T_k$ or $\dim T_k +1$.
    Thus we get, for each $1\leq k \leq n-1$,
    \begin{equation}\label{Tk/rel}
        \dim T_{k+1} -1 \leq \dim T_k \leq \dim T_{k+1} +R.
    \end{equation}
    Since $\dim T_n=\dim V_{\boldsymbol{F}}^*$ and $0\leq \dim T_k \leq k$, by \eqref{Tk/rel},
    it is easy to show \eqref{Tk/upp/bound} holds.
    Finally, by \eqref{fir/step}, \eqref{sec/step} and \eqref{Tk/upp/bound}, we obtain
    \begin{equation*}
        \codim V_{\boldsymbol{G},1}^*=2n-\dim V_{\boldsymbol{G},1}^* \geq 2n-n-\frac{Rn+\dim V_{\boldsymbol{F}}^*}{R+1}=\frac{\codim V^*_{\boldsymbol{F}}}{R+1}.
    \end{equation*}

Finally it follows by symmetry that the same bound holds for $\codim V_{\boldsymbol{G},2}^*$.  
This completes the proof.
\end{proof}

The next lemma is just \cite[Lemma 7.1]{LiuXie}.

\begin{lemma}\label{Lem/7/1}
    Let $s_d(n, \boldsymbol{F})$ be as in \eqref{def sd} for all $d$. Then 
    \begin{align*}\label{s_d upper bound}
        s_1(n, \boldsymbol{F})\le A_1(n, \boldsymbol{F}),
    \end{align*}
    where
    \begin{equation*}\label{def A1n}
        A_1(n, \boldsymbol{F}):=\frac{2^{D-1}(D-1)R}{n-\dim V^{\ast}_{\boldsymbol{F}}(n)}
    \end{equation*}  
    and $D$ is the highest degree in $\boldsymbol{F}$. 
\end{lemma}
We now propose the following lemma for comparing the singularities of systems $\boldsymbol{F}$ and $\boldsymbol{L}$, defined in Theorem \ref{prop gen Gau}.
\begin{lemma}\label{dim/rel}
Let $\boldsymbol{F}$ and $\boldsymbol{L}$ be as in Theorem \ref{prop gen Gau}.
Then
\begin{equation*}
\codim V^*_{\boldsymbol{L}} \geq \frac{\codim V^*_{\boldsymbol{F}}}{R+1},
\end{equation*}
where $R$ is the number of equations.
\end{lemma}
\begin{proof}
    Define
    \begin{equation*}
     V^*_{\boldsymbol{L},1}=\{(\boldsymbol{h},\boldsymbol{h}',\boldsymbol{j},\boldsymbol{j}') \in \mathbb{A}^{4n}:  
     \rank (J_{\boldsymbol{L},1})<R   \}
    \end{equation*}
    where $J_{\boldsymbol{L},1}$ consists of the first $2n$ columns of $J_{\boldsymbol{L}}$.
    Then $\dim V^*_{\boldsymbol{L}}\leq \dim V^*_{\boldsymbol{L},1}$, and therefore 
    \begin{align*} 
     \codim V^*_{\boldsymbol{L}}=4n-\dim V^*_{\boldsymbol{L}}\geq 4n-\dim V^*_{\boldsymbol{L},1}
     =\codim V^*_{\boldsymbol{L},1}.  
    \end{align*} 

    For $\boldsymbol{G}$ as in Theorem \ref{prop gen Gau}, we can define similarly $J_{\boldsymbol{G}, 1}$ and 
    $V^*_{\boldsymbol{G}, 1}$. Then, by the argument leading to \cite[(14)]{Yam19} and Proposition \ref{prop/dim}, 
    \begin{align*} 
     \codim V^*_{\boldsymbol{L},1} \geq \codim V^*_{\boldsymbol{G},1} \geq \frac{\codim V^*_{\boldsymbol{F}}}{R+1}. 
    \end{align*} 
 This completes the proof.
\end{proof}

\begin{proof}[Proof of Corollary \ref{thm2}]
Now we immediately deduce Corollary \ref{thm2} from Theorem \ref{prop gen Gau}, 
Lemmas \ref{Lem/7/1} and \ref{dim/rel}, and $\dim V^{\ast}_{\boldsymbol{F}}(n)\le R$, 
provided that $\boldsymbol{F}$ is nonsingular.
\end{proof}

\section{Application to the Birch--Goldbach problem}  

For a system $\boldsymbol{F} = (F_1, \ldots, F_R)$ of forms $F_i \in \mathbb{Z}[x_1, \ldots, x_n]$ with differing degrees, 
the Birch–Goldbach problem concerns the solubility of \eqref{bfF=bf0} in primes. 
Let $\mathfrak{B}$ be a fixed box in $n$-dimensional space defined by
\begin{align*}
b_i'< x_i\le b_{i}'',
\end{align*}
where $0<b_i'<b_{i}''<1$ are fixed constants for $i=1, \ldots, n$.  
We establish an asymptotic formula for the counting function  
\begin{align*}
N_{\boldsymbol{F}}(X) 
= \sum_{\substack{\boldsymbol{x} \in X\mathfrak{B} \\ \boldsymbol{F}(\boldsymbol{x}) = \mathbf{0}}} 
\Lambda(\boldsymbol{x}), 
\end{align*}
where $\Lambda(\boldsymbol{x}) = \Lambda(x_1)\cdots \Lambda(x_n)$ and $\Lambda(\cdot)$ is the von Mangoldt function. 
This yields a local-global principle for \eqref{bfF=bf0} in primes.  

\begin{theorem}\label{thm:main}
Let $F_1, \ldots, F_R \in \mathbb{Z}[x_1, \ldots, x_n]$ be forms with differing degrees, $D$ the highest degree, 
and $\mathcal{D}$ the sum of all degrees. If $\boldsymbol{F} = (F_1, \ldots, F_R)$ is nonsingular and 
\begin{equation*}\label{s/con/new}
n \geq D^2 4^{D+2} R^5, 
\end{equation*}
then
\[
N_{\boldsymbol{F}}(X) \sim \mathfrak{S}_{\boldsymbol{F}} \mathfrak{I}_{\boldsymbol{F}} X^{n - \mathcal{D}},
\]
where $\mathfrak{S}_{\boldsymbol{F}}$ and $\mathfrak{I}_{\boldsymbol{F}}$ are the singular series and singular integral 
associated with \eqref{bfF=bf0} in primes, both absolutely convergent. 
\end{theorem}

This improves upon \cite[Theorem~1.2]{LiuXie}, which required $n \geq D^2 4^{D+6} R^5$. 

The proof follows \cite{LiuXie}, so we only highlight the differences. The circle method begins with
\[
N_{\boldsymbol{F}}(X) = \int_{(0,1]^R} S_{\boldsymbol{F}}(\boldsymbol{\alpha})  d\boldsymbol{\alpha} 
\]
where 
\begin{equation*}\label{def/SFa}
S_{\boldsymbol{F}}(\boldsymbol{\alpha}) 
= \sum_{\boldsymbol{x} \in X\mathfrak{B}} \Lambda(\boldsymbol{x})
e\bigg(\sum_{i=1}^{R} \alpha_{i} F_{i}(\boldsymbol{x}) \bigg). 
\end{equation*}
The cube $(0,1]^R$ is partitioned into major arcs $\mathfrak{M}$ and minor arcs $\mathfrak{m}$ as in \cite{LiuXie}.  
Let 
\begin{align}\label{Def/vpi/Q} 
Q=X^{\frac{1}{4(R+1)}}.  
\end{align}
The major arcs are defined as 
\begin{equation*}\label{define MQ}
\mathfrak{M}=\mathfrak{M}(Q)=\bigcup_{1\le q\le Q}\bigcup_{\substack{1\le a_1,\ldots,a_R\le q
     \\ (a_1,\ldots,a_R,q)=1}}\mathfrak{M}(q,\boldsymbol{a};Q),
\end{equation*}
where
\begin{align*}
    \mathfrak{M}(q,\boldsymbol{a};Q)
=\bigg\{\boldsymbol{\alpha} 
\in \R^R:\ \bigg|\alpha_{i}-\frac{a_{i}}{q}\bigg|\leq \frac{Q}{qX^{\deg F_i}} \bigg\}. 
\end{align*}
The minor arcs are the complement of $\mathfrak{M}$ in $(0,1]^{R}$.

Under \eqref{Def/vpi/Q} and 
\begin{align}\label{Lem81/n>}
n\geq D^24^{D+2}R^5,  
\end{align}
we have, by \cite[Lemma 7.5]{LiuXie}, that 
\begin{align}\label{min/est}
\int_{\mathfrak{m} } S_{\boldsymbol{F}}(\boldsymbol{\alpha})  d\boldsymbol{\alpha} = o(X^{n-\mathcal{D}}). 
\end{align}
Note that $Q$ must be a positive power of $X$ as in \eqref{Def/vpi/Q}; the classical choice $Q=\log^B X$ is insufficient.

With $Q$ as in \eqref{Def/vpi/Q}, the major arcs are rather large, and we apply the saving-transfer method as 
summarized in \cite{LiuXie} to overcome the difficulties caused by the inapplicability of the Siegle--Walfisz theorem. 
The core of the method transfers savings from finite places to the infinite place, which is essential for 
handling systems with prime variables and differing degrees. For an exposition of the saving-transfer method, 
the reader is referred to \cite{LiuXieSavTra}.

\begin{lemma}\label{lem/bound/gauss}
Let $F_1, \ldots, F_R \in \mathbb{Z}[x_1, \ldots, x_n]$ be forms with differing degrees, 
and $D$ be the highest degree, and suppose that $\boldsymbol{F} = (F_1, \ldots, F_R)$ is nonsingular.  
Let \( \chi_1, \dots, \chi_n \) be Dirichlet characters modulo \( k_1, \dots, k_n \) respectively,  
where each \( k_i \) divides \( q \).  
Let \( k_0 = [k_1, \dots, k_n] \) be the least common multiple of the moduli \( k_1, \dots, k_n \), 
and let \( \chi^0 \) denote the principal character modulo \( q \).
Define   
\begin{align}\label{def:newNu}
\nu(q; \chi_1, \ldots, \chi_n)
= \sum_{\substack{1 \leq \boldsymbol{a} \le q \\ (a_1, \ldots, a_R, q) = 1}} 
\sum_{\boldsymbol{h} \bmod{q}} \bar{\chi}_1\chi^0(h_1) \cdots \bar{\chi}_n\chi^0(h_n)
e\bigg( \frac{\boldsymbol{a} \cdot \boldsymbol{F}(\boldsymbol{h})}{q} \bigg).
\end{align}
If 
\begin{align}\label{def:sLX}
n \geq D^24^{D+2}R^5, 
\end{align}
then there exists a constant $\delta>1$ such that 
\begin{equation*}
\sum_{\substack{q \leq Q \\ k_0 \mid q}} \frac{1}{\varphi^n(q)} |\nu(q; \chi_1, \dots, \chi_n)| \ll k_0^{-\delta} \log^n Q.  
\end{equation*}
\end{lemma}

This improves \cite[Lemma 8.1]{LiuXie}, which required $n\ge D^2 4^{D+6}R^5$. The improvement stems from the new bound 
for multiple Gauss sums in Theorem~\ref{prop gen Gau} and 
the general dimensional relationships between specific projective varities in Lemma~\ref{dim/rel}. 

\begin{proof} 
The inner sum over $\boldsymbol{h}$ in \eqref{def:newNu} equals 
$C_{\boldsymbol{F}}(q, \boldsymbol{a}; \bar{\chi}_1\chi^0,\ldots, \bar{\chi}_n\chi^0)$ 
as in \eqref{exp with chi}.  
We next employ the argument similar to that between Lemma 8.2 and 8.3 in \cite{BroHB}.
Set $d_j=\gcd \, (q, \boldsymbol{a}^{(j)},\ldots,\boldsymbol{a}^{(D)})$ for each $j \in \Delta$.
Suppose that $j_0$ is the least index $j \in \Delta$.
Then $d_{j_0}=1$ since $\gcd \, (q, \boldsymbol{a})=1$.
Moreover we have $d_j|q$ for every $j \in \Delta$.
The number of $\boldsymbol{a}^{(j)} (\bmod \, q)$ associated to a given $d_j$ is $(\frac{q}{d_j})^{r_j}$.
And the total number of $d_1,\ldots,d_D$ associated to a given $q$ is at most $\tau(q)^D\ll q^\varepsilon$.
Next we note that
$$
\min_{j\in \Delta} \bigg( \frac{d_j}{q}\bigg) ^{\frac{1}{4s_{2j_0}(4n,\boldsymbol{L})}} \leq \prod_{j \in \Delta} \bigg( \frac{d_j}{q}\bigg)^{\frac{\lambda_j}{s_{2j_0}(4n,\boldsymbol{L})}}
$$
for all 
$s_{2j_0}(4n,\boldsymbol{L})$ as in Theorem \ref{prop gen Gau} and 
for nany real numbers $\lambda_j\geq 0$ such that $\sum_{j \in \Delta}\lambda_j=\frac{1}{4}.$
We will apply this with
$$
\lambda_j:=
\begin{cases}
   \theta+r_{j_0}s_{2j_0}(4n,\boldsymbol{L}), &\mbox{if $j=j_0$,} \\
   r_js_{2j}(4n,\boldsymbol{L}), &\mbox{if $j \in \Delta \setminus \{ j_0\}$,}
\end{cases}
$$
where 
$$
\theta=\frac{1}{4}-\sum^D_{i=1}s_{2i}(4n,\boldsymbol{L})r_i. 
$$ 
We claim that \eqref{def:sLX} implies 
\begin{equation}\label{</claim}
s_{2j_0}(4n,\boldsymbol{L})+\sum^D_{i=1}s_{2i}(4n,\boldsymbol{L})r_i<\frac{1}{4}.
\end{equation}
It follows from Theorem \ref{prop gen Gau} that
\begin{align*} 
\nu(q; \chi_1, \dots, \chi_n)
& \ll q^{n+\varepsilon} \sum_{d_1,\ldots,d_D|q} \bigg(\frac{1}{q}\bigg)^{\frac{\theta}{s_{2j_0}(4n,\boldsymbol{L})}} \prod_{j\in \Delta} \bigg(\frac{q}{d_j} \bigg)^{r_j} \bigg( \frac{d_j}{q} \bigg)^{r_j}  \\ 
& \ll q^{n-\frac{\theta}{s_{2j_0}(4n,\boldsymbol{L})}+\varepsilon} \ll q^{n-\delta} 
\end{align*} 
with $\delta>1$ a constant, proving the desired result.

Now we prove the claim that \eqref{def:sLX} implies \eqref{</claim}. 
By Lemma \ref{Lem/7/1},
\begin{equation*}
s_{2C}(4n,\boldsymbol{L})+\sum^D_{i=1}s_{2i}(4n,\boldsymbol{L})r_i \leq (R+1)\frac{2^{2D-1}(2D-1)R}{4n-\dim V^*_{\boldsymbol{L}}(4n)}.
\end{equation*}
Then, by Lemma \ref{dim/rel}, 
\begin{equation*}
 s_{2C}(4n,\boldsymbol{L})+\sum^D_{i=1}s_{2i}(4n,\boldsymbol{L})r_i \leq \frac{2^{2D-1}(2D-1)(R+1)^2R}{n-\dim V^*_{\boldsymbol{F}}(n)}.
\end{equation*}
By \eqref{def:sLX} and $\dim V^*_{\boldsymbol{F}}(n)\leq R$, we get
$$
n\geq D^24^{D+2}R^5> 2^{2D+1}(2D-1)(R+1)^2R+\dim V^*_{\boldsymbol{F}}(n), 
$$
and hence 
$$ 
s_{2C}(4n,\boldsymbol{L})+\sum^D_{i=1}s_{2i}(4n,\boldsymbol{L})r_i<\frac{1}{4}, 
$$ 
proving the claim. 
\end{proof}

Using this lemma in place of \cite[Lemma~8.1]{LiuXie}, we obtain
\[
\int_{\mathfrak{M} } S_{F}(\alpha) d\alpha
\sim 
 \mathfrak{S}_{F} \mathfrak{J}_{F} X^{n-\mathcal{D}} 
\]
under \eqref{def:sLX}, the same as \eqref{Lem81/n>}. 
Combined with \eqref{min/est} under \eqref{Lem81/n>}, this proves Theorem~\ref{thm:main}. 

\section*{Acknowledgments}
The authors thank Yang Cao, Lei Fu, Philippe Michel, Daqing Wan, Shuntaro Yamagishi, and Dingxin Zhang for 
helpful discussions on multiple exponential sums and algebraic geometry. 
The authors are grateful to Kien Huu Nguyen for pointing out an issue arising in the application of \cite{Ngu21} in an earlier version of this manuscript. 
In the present version, we avoid reliance on \cite{Ngu21}.  
This work was supported by the National Key Research and Development Program of China (No. 2021YFA1000700) 
and the National Natural Science Foundation of China (No. 12031008).

\end{document}